\documentclass{amsart}

\usepackage{amsmath}
\usepackage{amssymb}
\usepackage{amsfonts}
\usepackage{enumerate}
\usepackage{vmargin}

\newcommand{\tens}{\otimes}
\newcommand{\bC}{{\mathbb{C}}}

\newcommand{\bN}{{\mathbb{N}}}
\newcommand{\bQ}{{\mathbb{Q}}}
\newcommand{\bR}{{\mathbb{R}}}
\newcommand{\bT}{{\mathbb{T}}}
\newcommand{\bZ}{{\mathbb{Z}}}

  \newcommand{\A}{{\mathcal{A}}}

  \newcommand{\E}{{\mathcal{E}}}

  \newcommand{\M}{{\mathcal{M}}}

\renewcommand{\P}{{\mathcal{P}}}

\renewcommand{\phi}{\varphi}
\newcommand{\upchi}{{\raise.35ex\hbox{\ensuremath{\chi}}}}

\renewcommand{\leq}{\leqslant}
\renewcommand{\geq}{\geqslant}

\newcommand{\spn}{\operatorname{span}}

\newtheorem{thm}{Theorem}[section]

\newtheorem{prop}[thm]{Proposition}
\newtheorem{cor}[thm]{Corollary}
\newtheorem{lemma}[thm]{Lemma}
\newtheorem{rk}[thm]{Remark}

\begin{document}

\title{$L_p$-Multipliers on Quantum Tori}
\date{}
\author[\'E. Ricard]{\'Eric Ricard}
\address{Laboratoire de Math{\'e}matiques Nicolas Oresme,
Universit{\'e} de Caen Normandie,
14032 Caen Cedex, France}
\email{eric.ricard@unicaen.fr}

\thanks{{\it 2010 Mathematics Subject Classification:} 46L51; 47A30.} 
\thanks{{\it Key words:} Noncommutative $L_p$-spaces}

\begin{abstract}
It was shown by Chen, Xu and Yin that completely bounded Fourier
multipliers on noncommutative $L_p$-spaces of quantum tori
$\bT^d_\theta$ do not depend on  the parameter $\theta$. We
establish that the situation is somehow different for bounded multipliers. The
arguments are based on transference from the commutative torus.
\end{abstract}
\maketitle
\section{Introduction}

 Quantum tori are very natural examples in noncommutative
 geometry for a long time. Nevertheless not much was done about them
 from a more analytical point of view. Elements in quantum tori
 have a formal Fourier expansion. Thanks to it most of the classical
 objects from harmonic analysis make sense and one is tempted to
 look for counterparts of basic results. This approach really started
 with a recent paper \cite{CXY} by Chen, Xu and Yin.  Motivated by the
 developments of noncommutative harmonic analysis and
 integration, they established for instance the pointwise convergence of
 Fej\'er or Bochner-Riesz means in $L_p$. They also considered natural
 Fourier multipliers and showed that their complete bounded norm on
 $L_p$ do not depend on the parameters of the tori. In another article
 \cite{XXY} , Xiong, Xu and Yin continued this line of investigation
 by looking at analogues of classical function spaces. Fourier
 multipliers and transference techniques played an important r\^ole
 there. The main goal of this note is to show that the bounded norm of
 $L_p$-multipliers on noncommutative tori somehow depends on the parameters
 solving a question from \cite{CXY}.

 We refer to \cite{PX, P2} for results on noncommutative $L_p$-spaces associated to (semi-finite) von Neumann algebras and to \cite{Pbook} for operator spaces.
 We recall the basics on quantum tori, we refer to \cite{Dav, CXY} for
 more details. Let $d\geq 1$ and $\theta$ be a real skew symmetric matrix
 of size $d$. The $d$-dimensional noncommutative torus associated to
 $\theta$ is the universal $C^*$-algebra $\A_\theta^d$ generated by
 $d$ unitary operators $U_k^\theta$ such that
$$U_k^\theta U_j^\theta= e^{2\pi i \theta_{k,j}} U_j^\theta
 U_k^\theta.$$ We may drop the exponent $\theta$ when no confusion can
 occur. Also when $d=2$, we prefer to think of $\theta$ as a real
 number and we use $U_\theta$ and $V_\theta$ for the generators. So
 that we have $U_\theta V_\theta=e^{2i\pi \theta}V_\theta U_\theta$.

 For a multi-index $m=(m_1,...,m_d)\in \bZ^d$, we put
 $U^m=U_1^{m_1}...U_d^{m_d}$ for basis monomials. As usual a polynomial
 is a finite sum of monomials. The set of polynomials is dense in
 $\A_\theta^d$.  There is a natural normalized trace $\tau$ on
 $\A_\theta^d$ defined by $\tau(U^m)=0$ unless $m=0$. The GNS
 representation of $\tau$ gives rise to a finite von Neumann algebra
 $L_\infty(\bT^d_\theta)$. We denote by $L_p(\bT^d_\theta)$ the
 associated $L_p$-spaces. Any element in $x\in L_1(\bT^d_\theta)$ has
 a natural formal Fourier expansion $x\sim \sum \hat x (m) U^m$ where
 $\hat x (m)=\tau(x(U^m)^*)$. Given a function $\phi: \bZ^d \to \bC$,
 we say that it is the symbol of a multiplier (resp. completely
 bounded) on $L_p(\bT^d_\theta)$ if the map defined on monomials by
 $M_\phi(U^m)=\phi(m)U^m$ extends to bounded (resp. completely
 bounded) map on $L_p(\bT^d_\theta)$.  We refer to \cite{CXY,XXY} for
 basic facts on those multipliers. We will use quite often multipliers
 of Fej\'er type $F_n^d :\bZ^d\to \bC$ given by
 $F_n^d(m_1,...,m_d)=\prod_{i=1}^d\Big( 1-\frac {|m_i|}n\Big)^+$. They
 give rise to unital completely positive multipliers and $M_{F^d_n}$
 is an completely contractive pointwise approximation of the identity
 in $L_p(\bT^d_\theta)$, $1\leq p<\infty$ and in $\A_\theta^d$ (see Proposition \ref{cbnor}).

When $\theta=0$, we recover classical function spaces
$L_p(\bT^0_d)=L_p(\bT^d,dm)$ where $dm$ is the Haar measure on $\bT^d$
and $\A_0^d=C^*(\bZ)$.  In \cite{P2} Proposition 8.1.3, Pisier gave
the first example of a bounded multiplier on $L_p(G)$ ($1<p\neq 2<
\infty$ and $G$ any infinite compact abelian group) that is not completely
bounded. We aim to use it for Fourier multipliers on
$L_p(\bT^d_\theta)$. The idea behind it is pretty simple : when
$\theta\notin \bQ$, $\A^d_\theta$ contains asymptotically a nice copy
of $C^*(\bZ)\tens \A^d_\theta$.

 The paper is organized as follows. First we give a short proof of the
 independences on $\theta$ of completely bounded multipliers in section
 2. Next, we look at bounded $L_p$-multipliers ($1<p\neq 2< \infty$),
 we explain the basic commutative estimates we need, the construction
 of ultraproducts and derive the result. We finish by looking at
 bounded $L_\infty$-multipliers.

\section{Completely bounded multipliers}

 In this section, we give a short direct proof of the following result
 from \cite{CXY}:
\begin{prop}\label{cbnor}
Let $1\leq p\leq \infty$, for any $\phi:\bZ^d \to \bC$, we have 
$$\|M_\phi \|_{cb(L_p(\bT_\theta^d))}=\|M_\phi\|_{cb(L_p(\bT^d))}.$$
\end{prop}
\begin{proof}
By the universal property of
$\A_{\theta+\theta'}^d$, one can define a $*$-homomorphism
$\pi:\A_{\theta+\theta'}^d \to \A_\theta^d \tens \A_{\theta'}^d$ by
$\pi(U_i^{\theta+\theta'})=U_i^\theta\tens U_i^{\theta'}$. It is
obviously trace preserving and hence extends to 
a complete isometry on all $L_p$, $1\leq p \leq \infty$.
Since $(Id\tens M_\phi) \circ \pi= \pi\circ M_\phi$, it follows
that $\|M_\phi\|_{cb(L_p(\bT_{\theta+\theta'}^d))}\leq \|M_\phi\|_{cb(L_p(\bT_{\theta}^d))}$
for any $\theta$ and $\theta'$ from which equality follows. 
\end{proof}

 For any $\phi:\bZ^d \to \bC$, let $s_\phi$ be the (possibly
 unbounded) linear functional functional on $\A_\theta^d$ be given on monomials
by  $s_\phi(U^{m})=\phi(m)$, $m\in \bZ^d$.

 The preceding proposition can be extended to 
\begin{prop}\label{cbnor2}
Let $1\leq p\leq \infty$, for any $\phi:\bZ^d \to \bC$, we have 
$$\|M_\phi: L_p(\bT_\theta^d)\to L_p(\bT_\gamma^d)\|_{cb}=\|M_\phi : L_p(\bT_{\theta-\gamma}^d)\to L_p(\bT^d)\|_{cb}.$$
Moreover
$$\|M_\phi: L_\infty(\bT_\theta^d)\to L_\infty(\bT_\gamma^d)\|_{cb}=\|s_\phi\|_{(\A_{\theta-\gamma}^d)^*}.$$
\end{prop}
\begin{proof}
The proof of the first part is the same as above. For the second one,
assume $M_\phi : L_\infty(\bT_{\theta-\gamma}^d)\to
L_\infty(\bT^d)$. Composing it with a Fej\'er type kernel $F^d_n$,
$M_{F^d_n\phi}: L_\infty(\bT_{\theta-\gamma}^d) \to C^*(\bZ^d)$. Composing
again with the trivial character, we get that
$\|s_{F^d_n\phi}\|_{(\A_{\theta-\gamma}^d)^*}\leq \|M_\phi\|$. Letting
$n\to \infty$ gives the first part. For the reverse inequality, if
$\|s_\phi\|<\infty$ then by composition with $M_{F_n^d}$, $\|s_{F^d_n\phi}\|\leq
\|s_\phi\|$, using the $*$-homomorphism $\pi:\A_{\theta}^d \to
\A_\gamma^d \tens \A_{\theta-\gamma}^d$ and composing it with $Id\tens
s_{F^d_n\phi}$, we get that $\|M_{F^d_n\phi} :L_\infty(\bT_{\theta}^d)\to 
L_\infty(\bT_{\gamma}^d)\|\leq \|s_\phi\|$. But $M_{F^d_n\phi}$ is then a family of normal uniformly bounded
maps converging pointwise for the $w^*$-topology. Its limit is $M_\phi$.
\end{proof}
\section{Bounded multipliers on $L_p$, $1<p<\infty$}

\subsection{Periodization of Fourier multipliers}

We aim to prove a periodization result for Fourier multipliers on
$\bT$. We proceed by using some folklore discretization arguments.

Denote by $z$ the canonical generator of $L(\bZ)=L_\infty(\bT)$ and for any 
$n\geq 1$, $\gamma_n$ that of $L(\bZ/n\bZ)$ the group von Neumann algebra associated to the finite group $\bZ/n\bZ$. We always choose the representative
of a class in $\bZ/n\bZ$ with smallest absolute value.

\begin{lemma}\label{multcor}
Let $f:\bR^+\to \bR^+$ be a non decreasing convex function with $f(0)=1$,
 then for any
$n\geq 0$, there exists a measure $\mu$ on $\bT$ such that
$$ \forall\, |k|\leq n, \;\hat \mu (k)= f(|k|) \qquad \textrm{ and }
\qquad \| \mu\| \leq f(n)^2.$$    
\end{lemma}
\begin{proof}
The sequence given by $m_k=f(n-k)$ for $k\leq n$ and $m_k=f(0)=1$ for
$k>n$ is non increasing convex, hence there exists a positive measure
$\nu$ with $\hat \nu(k)=m_{|k|}$ with mass $f(n)$, see Theorem 3.7 in
\cite{Cha}. Then $\mu=(z^n\nu)*(z^{-n}\nu)$ has the desired
properties.
\end{proof}

 We fix some $1\leq p\leq \infty$.
We also denote by 
$\P_d=\spn \{z^k; |k|\leq d\}\subset L_p(L(\bZ))$ and similarly $\P_d^n=\spn \{(\gamma_n)^k; |k|\leq d\}\subset L_p(L(\bZ/n\bZ))$.

For $2d<n$, one can consider the formal identity $j_{d,n}:\P_d\to \P_{d,n}$ 
given by $j_{d,n}(z^k)=\gamma_n^k$.

As usual let $\sin_c(x)$ be $\frac {\sin x}x$.

\begin{prop}\label{discest}For any $n>2d$, 
$$ \| j_{d,n} \|_{cb}\leq \sin_c^{-2}\big(\frac{d\pi}n\big), \qquad 
\| j_{d,n}^{-1} \|_{cb}\leq \Big((1- \frac {2d}n)\sin_c\big(\frac{d\pi}n\big)\Big)^{-2}.$$
\end{prop}

\begin{proof}
One can see $L_p(L(\bZ/n\bZ))$ as a subspace of $L_p(L(\bZ))=L_p(\bT)$
by looking at piecewise constant functions. More precisely, put for
$z\in \bT$, $g(z)=\omega^l$ where $\omega^l$ is the closest $n$-root
of 1 to $z$, then $g$ and $\gamma_n$ have the same distribution. We
identify them. 

Let $\E$ be the conditional expectation
onto the algebra generated by $g\approx \gamma_n$. Easy computations give that
for $k\in \bZ$, $\E z^k=  \sin_c\big(\frac {k\pi}n\big) (\gamma_n)^k $.
The function $\sin_c^{-1}$ is convex on $[0,\pi[$ , using Lemma
    \ref{multcor}, there exists a measure on $\bT$ with $\mu (k)=
    \sin_c^{-1}(\frac{|k|\pi} n)$ for $|k|\leq d$ with mass
at most $\sin_c^{-2}(\frac{d\pi}n)$. For any $f\in \P_d$, we have 
$j_{n,d}(f)=\E(f*\mu)$, this gives the first estimate.

Computing the Fourier expansion in $L_2(\bT)$ gives
for $|k|<n$
$$(\gamma_n)^k= \sum_{j\in \bZ}   \sin_c\Big(\frac{\pi(k+jn)} n\Big) z^{k+jn}.$$
Hence with the Fej\'er kernel $F_{n/2}$, assuming $|k|\leq d<n/2$ :
$$F_{n/2}*\gamma_n^k=\Big(1- \frac {2|k|}n\Big)\sin_c\Big(\frac{\pi k} n\Big) z^k.$$
To get $j_{n,d}^{-1}$ one just need to make  corrections as above using that 
$\sin_c^{-1}$ and $(1-x)^{-1}$ are convex.
\end{proof}

\begin{prop}\label{permult}
Let $\phi:\bZ\to \bC$ be a Fourier multiplier on $L_p(\bT)$, then
there exists a sequence of periodic multipliers $\phi_n:\bZ\to \bC$ converging pointwise to $\phi$ with 
$\|M_{\phi_n}\|\leq\|M_{\phi}\|$ and $\|M_{\phi_n}\|_{cb}\leq\|M_{\phi}\|_{cb}$.
\end{prop}

\begin{proof}
Fix $n>2d$. We use transference once again. Consider the measure
preserving $*$-homo\-mor\-phism $\pi_n: L(\bZ)\to L(\bZ\tens \bZ/n\bZ)$
given by $\pi_n(z)=z\tens \gamma_n$. Let $\psi:\bZ/n\bZ\to \bC$ and $q:
\bZ\to \bZ/n\bZ$ be the quotient mapping. Using the intertwining
identity $(Id\tens M_\psi)\circ \pi_n= \pi_n\circ M_{\psi\circ q}$ and
the fact that bounded maps on $L_p$ tensorize with the identity of
$L_p$ by Fubini's theorem, one gets that
$$\| M_{\psi\circ q} : L_p(\bT) \to L_p(\bT)\|\leq \| M_{\psi} : L_p(\bT/n\bZ) \to L_p(\bT/n\bZ)\|.$$
Note that, for $d<n$, the Fej\'er Kernel $F_d$ is also positive definite on
$\bZ/n\bZ$.  The map $j_{n,d}\circ M_\phi\circ j_{n,d}^{-1}\circ
M_{F_d}$ is the Fourier multiplier $\tilde \phi_{n,d}$ on $\bZ/n\bZ$
so that $M_{\tilde
  \phi_{n,d}}(\gamma_n^k)=\phi(k)\big(1-\frac{|k|}d)\gamma_n^k$. Using
Proposition \ref{discest}, we get $\|M_{\tilde \phi_{n,d}}\|\leq c_{n,d}
\|M_{\phi}\|$ with $c_{n,d}\to 1$ as $n/d\to \infty$ and
similarly with cb norms.
 
 Let $\phi_{n}$ be $\tilde
 \phi_{n^2,n}/c_{n^2,n} \circ q$.  Obviously $\|M_{\phi_{n}}\|\leq
 \|M_{\phi}\|$, $\|M_{\phi_{n}}\|_{cb}\leq \|M_{\phi}\|_{cb}$ and
 $\lim_{n\to \infty} \phi_n(k)=\phi(k)$.
\end{proof}

\subsection{Ultraproducts}

 Given a family of finite von Neumann algebras $(\M_i,\tau_i)_{i\in
   \bN}$ with normalized trace $\tau_i$, their (reduced) von Neumann
 ultraproduct along the free ultrafilter $\mathfrak U$ is
$$\M_{\mathfrak U}= \ell_\infty(\M_i)/I_{\mathfrak U},
\quad  I_{\mathfrak U}=\{(x_i)\in \ell_\infty(\M_i) \,|\, \lim_{\mathfrak U} \tau_i(x_i^*x_i)=0\}.$$
It is a finite von Neumann algebra with normalized trace $\tau((x_i)^\bullet)=\lim_{\mathfrak U} \tau_i(x_i)$, see \cite{Pbook} section 9.10.

\begin{lemma}\label{ult}
Let $T_i:\M_i\to\M_i$ be a sequence of maps such that $\sup_i
\|T_i\|_{B(\M_i)}+ \|T_i\|_{B(L_2(\M_i))}<\infty$, then
$T_{\mathfrak U}=\prod_{\mathfrak U} T_i$ is well defined on
$\M_{\mathfrak U}$ and moreover $\|T_{\mathfrak U}\|_{B(L_p(\M_{\mathfrak U}))}\leq
\lim_{\mathfrak U}\|T_i\|_{B(L_p(\M_i))}$ (the same holds for cb norms).
\end{lemma}

\begin{proof}
The assumptions yield that $T_{\mathfrak U}$ is well defined on $\ell_\infty(\M_i)$ 
and stabilizes $I_{\mathfrak U}$, so it is well defined on $\M_{\mathfrak U}$.
 
For $x=(x_i)^\bullet\in \M_{\mathfrak U}$, we have
$\|x\|_{L_p(\M_{\mathfrak U},\tau)}= \lim_i \|x_i\|_{L_p(\M_i,\tau_i)}$
by the definition of the trace. The second assertion follows.
\end{proof}

We now focus on 2-dimensional tori $\bT_\theta^2$. Recall that the two
generators satisfy $U_\theta V_\theta=e^{i2\pi\theta} V_\theta U_\theta$.  We still
denote by $z$ the generator of $L_\infty(\bT)$.

\begin{prop}\label{emb}
Fix some integer $N>0$. For $\theta\notin\bQ$ and $\gamma\in \bR$, there is an explicit cb-isometric embedding 
 $$\pi : L_p(\bT)\otimes_p L_p(\bT_\gamma^2) \to
L_p(\bT^2_\theta)_{\mathfrak U},$$
obtained by a *-representation of the form $$\pi(z)=(U_\theta)^\bullet,\quad \pi(U_\gamma)=(U_\theta^{Nl_n})^\bullet, \quad
\pi(V_\gamma)=(V_\theta^{k_n})^\bullet.$$
\end{prop}
\begin{proof}
Let $n>0$, since $\theta \notin \bQ$, one can find some $k_n$ such that
$|e^{i2\pi k_n \theta}-1|<\frac 1n$. Similarly as $k_nN\theta\notin
\bQ$, one can find $l_n$ such that $|e^{i2\pi k_nNl_n \theta}-e^{i2\pi\gamma}|<\frac 1n$.
The above choices of $l_n$ and $k_n$ have been
made such that $\pi(z)$ commutes with $\pi(U_\gamma)$ and $\pi(V_\gamma)$ and
$\pi(U_\gamma)\pi(V_\gamma)=e^{i2\pi\gamma}\pi(V_\gamma)\pi(U_\gamma)$. Notice that 
 $C^*(\bZ)\tens \A_\theta^2$ is a 3-dimensional torus. By its universal
property, we can extend $\pi$ to a
$*$-representation $\pi:C^*(\bZ)\tens \A_\theta^2\to
L_\infty(\bT^2_\theta)_{\mathfrak U}$. It is obviously trace preserving, hence it 
extends to complete isometries at all $L_p$-levels
\end{proof}

\subsection{Applications to multipliers}

 \begin{prop}\label{multtens}
Let $1\leq p\leq \infty$, and $\phi:\bZ\to \bC$ be periodic and $\theta\notin \bQ$, then
$$\| M_{\phi\tens 1} : L_p(\bT_\theta^2)\to L_p(\bT_\theta^2) \| = \|M_\phi :L_p(\bT)\to L_p(\bT)\|_{cb}.$$
\end{prop}
\begin{proof}
The proof uses transference again. Let $N$ be the period of $\phi$.
We use the embedding $\pi$ of $L_p(\bT)\tens L_p(\bT^2_\theta)$ from Proposition \ref{emb} with $\gamma=\theta$.
 As $\phi$ is $N$-periodic, the
multiplier $M_{\phi\tens 1}$ is well defined and uniformly bounded on $L_\infty$ and $L_2$, we can consider its ultrapower $(M_{\phi\tens 1})_{\mathfrak U}$ by
 Lemma \ref{ult}.  Take $a,b,c\in\bZ$, we have by $N$-periodicity
 $$(M_{\phi\tens 1})_{\mathfrak U} \pi(z^aU_\theta^bV_\theta^c)=
(M_{\phi\tens
  1}(U_\theta^{a+bNl_n}V_\theta^{cl_n}))^\bullet=\phi(a)\pi(z^aU_\theta^bV_\theta^c).$$
 By linearity, continuity and density, we get for all $x\in
L_p(\bT)\tens L_p(\bT^2_\theta)$ $$\pi\Big((M_\phi\tens Id) x\Big)=
(M_{\phi\tens 1})_{\mathfrak U} \pi(x).$$ Hence by Lemma \ref{ult}, we get
that $\|M_\phi\tens Id\|\leq \|M_{\phi\tens 1}\|$.  But
$L_\infty(\bT_\theta^2)$ is the hyperfinite factor, so we can conclude
$\|M_\phi\tens Id\|=\|M_\phi\|_{cb}$. The other inequality is clear by
Proposition \ref{cbnor} as
$$\|M_\phi\|_{cb(L_p(\bT))}=\|M_{\phi\tens 1} \|_{cb(L_p(\bT^2))}=\|M_{\phi\tens 1} \|_{cb(L_p(\bT_\theta^2))}.$$
\end{proof}

\begin{thm}\label{ex1}
For any $\theta\notin \bQ$ and $1<p\neq2 <\infty$, there exists $\phi:\bZ^2
\to \bC$ such that
$$ \|M_\phi : L_p(\bT^2)\to  L_p(\bT^2)\| <\infty \qquad \textrm{ but }
\qquad  \|M_\phi : L_p(\bT^2_\theta)\to  L_p(\bT^2_\theta)\| =\infty.$$
\end{thm}

\begin{proof}
If not, by the closed graph theorem, there is a constant $c$ such that 
$$\|M_\phi : L_p(\bT^2_\theta)\to  L_p(\bT^2_\theta)\| \leq  c \|M_\phi : L_p(\bT^2)\to  L_p(\bT^2)\|.$$
 By Pisier's example, there is a multiplier $\phi$ on
$L_p(\bT)$ that is not completely bounded.  By Proposition
\ref{permult}, we can find periodic multipliers $\phi_n$ on $L_p(\bT)$
with $\|M_{\phi_n}\|\to \|M_\phi\|$ and $\|M_{\phi_n}\|_{cb}\to
\infty$.

Next $\|M_{\phi_n\tens 1} : L_p(\bT^2_\theta)\to L_p(\bT^2_\theta)\|\to
\infty$ by Proposition \ref{multtens}. But 
$\|M_{\phi_n\tens 1} : L_p(\bT^2)\to  L_p(\bT^2)\|\to \|M_\phi: L_p(\bT)\to  L_p(\bT)\|$ by Fubini's theorem.
\end{proof}

\begin{rk} Actually one can directly construct periodic $\phi_n$ from Pisier's example.
\end{rk}

\begin{rk}
It is also easy to see that $1<p\neq 2<\infty$ there exists $\phi:\bZ^2 \to
\bC$ such that
$$ \|M_\phi : L_p(\bT_\theta^2)\to  L_p(\bT_\theta^2)\| <\infty \qquad \textrm{ but }
\qquad  \|M_\phi : L_p(\bT^2_\theta)\to  L_p(\bT^2_\theta)\|_{cb} =\infty.$$
This also follows directly from Pisier's example as there is a 
conditional expectation from $L_\infty(\bT_\theta^2)$ to $L_\infty(\bT)$ the algebra generated by $U$.
\end{rk}

\begin{rk}
When $\theta=\frac a b\in \bQ$, then $L_p(\bT_\theta^2)\subset
L_p(\bT^2,S_p^b)$; a possible embedding is given by
$\pi(U_\theta)=b^{-1/p}z_1\sum_{k=1}^b e^{2i \pi\frac {ak} b} e_{k,k}$,
$\pi(V_\theta)=b^{-1/p}z_2\sum_{k=1}^b e_{k,k-1(b)}$ where $e_{k,l}$ is the
canonical basis of $S_p^b$. Hence multipliers are automatically
completely bounded. Nevertheless, the equivalence constant must behave
badly, but we have no quantitative estimates.
\end{rk}

\begin{rk}
The algebra $L_\infty(\bT^2_\theta)$ can be seen as a crossed product
$L_\infty(\bT)\rtimes_\theta \bZ$ where the action comes from the rotation of
angle $\theta$.  Thus, the above proofs produce a
$\bZ$-equivariant map $T$ on $L_p(\bT)$ such that $T\rtimes Id$ is unbounded on $L_p$. 
\end{rk}

\begin{cor}\label{ex2}
For any $\theta,\gamma\in \bR\backslash \bQ$ such that $\theta\notin \gamma \bQ$  and $1<p\neq 2<\infty$, there exists $\phi:\bZ^2 \to \bC$ such that 
$$ \|M_\phi : L_p(\bT_\gamma^2)\to  L_p(\bT_\gamma^2)\| <\infty \qquad \textrm{ but }
\qquad  \|M_\phi : L_p(\bT^2_\theta)\to  L_p(\bT^2_\theta)\| =\infty.$$
\end{cor}

\begin{proof}
This is once again a transference from the multiplier constructed in Theorem \ref{ex1}. 

As above if the result does not hold, there exists $c$ with 
\begin{equation}\label{estm}
\|M_\psi: L_p(\bT^2_\theta)\to L_p(\bT^2_\theta)\|\leq c \, \|M_\psi: L_p(\bT^2_\gamma)\to L_p(\bT^2_\gamma)\|.
\end{equation}
 By the assumption, for all $n>0$, we can find some $k_n$ such that 
$$|e^{i2\pi k_n \gamma}-1| <\frac 1 n,\qquad |e^{i2\pi k_n \theta}-e^{i2\pi \theta}| <\frac 1 n.$$
Indeed the set $\bN(\theta,\gamma)$ is equidistributed in $\bR^2$ modulo 1.

The subalgebra $\{U_\gamma,V_\gamma^{k_n}\}''$ is $L_\infty(\bT^2_{k_n
  \gamma})$, the associated conditional expectation corresponds to the
multiplier associated to the function $\psi(k,l)=1_{k_n|l}$. And
similarly with $\theta$. Thus applying \eqref{estm} to multipliers
of the form $\phi(k,l)=1_{k_n|l}f(k,l/k_n)$ for $f:\bZ^2\to \bC$, we deduce that 
\begin{equation}\label{estm2}\|M_f: L_p(\bT^2_{k_n\theta})\to L_p(\bT^2_{k_n \theta})\|\leq c \, \|M_f: L_p(\bT^2_{k_n \gamma})\to L_p(\bT^2_{k_n \gamma})\|.\end{equation}
 
Let $\P_d^\alpha=\spn\{ U_\alpha^kV_\alpha^l ; |k|,\, |l|\leq d\}$. Going to ultraproducts as in Proposition \ref{emb} we deduce
that there are cb isometric embeddings
\begin{equation}\label{emb2}
\pi_\gamma : L_p(\bT^2) \to \prod_{\mathfrak U} L_p(\bT^2_{k_n\gamma}),\quad 
\pi_\theta : L_p(\bT^2_\theta) \to \prod_{\mathfrak U}L_p(\bT^2_{k_n\theta})\end{equation}
given by the extension of 
$\pi_\gamma(U_0)=(U_{k_n\gamma})^\bullet$, $\pi_\gamma(V_0)=(V_{k_n\gamma})^\bullet$ and 
 $\pi_\theta(U_\theta)=(U_{k_n \theta})^\bullet$, $\pi_\theta(V_\theta)=(V_{k_n\theta})^\bullet$.

In particular this restricts to $(\P_d^0,\|.\|_p)=\prod_{\mathfrak U} (\P_d^{k_n\gamma},\|.\|_p)$ given by the ultraproduct of the formal 
identity maps $i_n: U_0^kV_0^l\mapsto U_{k_n\gamma}^kV_{k_n\gamma}^l$
(recall that the ultraproduct of $m$-dimensional spaces is still $m$-dimensional).
 As the spaces share the same finite dimension, by a compactness argument, 
we get that 
$\lim_{\mathfrak U} \|i_n : (\P_d^0,\|.\|_p)\to (\P_d^{k_n\gamma},\|.\|_p)\|=
\lim_{\mathfrak U} \|i_n^{-1} : (\P_d^{k_n\gamma},\|.\|_p)\to (\P_d^{0},\|.\|_p)\|=1$.

Let $\phi:\bZ^2\to \bC$ and $F^2_d=F_d\tens F_d$ be the symbol of the
2-dimensional Fej\'er kernel. As we can write $M_{F_d^2\phi}=i_n\circ
M_\phi\circ i_n^{-1}\circ M_{F^2_d}$, we get $\lim_{\mathfrak U} \|
M_{F_d^2\phi}: L_p(\bT^2_{k_n\gamma}) \to L_p(\bT^2_{k_n\gamma})\|\leq
\|M_\phi :L_p(\bT^2)\to L_p(\bT^2)\|$.

Using the isometric embedding $\pi_\theta$ from \eqref{emb2}, the
estimate \eqref{estm2}, we can consider the ultraproduct of
$M_{F_d^2\phi}$ to get $\|M_{F_d^2\phi}: L_p(\bT^2_\theta) \to
L_p(\bT^2_\theta)\|\leq c \|M_\phi :L_p(\bT^2)\to
L_p(\bT^2)\|$. Letting $d\to\infty$ contradicts Theorem \ref{ex1}.

\end{proof}

\section{Bounded multipliers on $L_\infty$}

 We conclude by giving an analogue of Theorem \ref{ex1} for $p=\infty$
 (or $p=1$ by duality), recall that $\|M_\phi : L_\infty(\bT^2_\theta)\to  L_\infty(\bT^2_\theta)\|_{cb}=\|M_\phi : L_\infty(\bT^2)\to  L_\infty(\bT^2)\|$ by Proposition 
\ref{cbnor}.
\begin{prop}\label{ex3}
For any $\theta\notin \bQ$, there exists $\phi:\bZ^2 \to \bC$ such that 
$$ \|M_\phi : L_\infty(\bT_\theta^2)\to  L_\infty(\bT_\theta^2)\| <\infty \qquad \textrm{ but }
\qquad  \|M_\phi : L_\infty(\bT^2_\theta)\to  L_\infty(\bT^2_\theta)\|_{cb} =\infty.$$
\end{prop}

\begin{proof}Assume the conclusion does not hold, then there is a constant $c$, with 
$\|M_\phi\|_{cb}\leq c\|M_\phi\|$.

One can construct inductively increasing sequences of integers $k_n,\,
l_n$ such that $k_n$ is odd and
$$\forall j<n,\qquad |e^{i2\pi \theta k_jl_n}+1|< \frac 1{2^n},\quad |e^{i2\pi \theta k_nl_j}-1|< \frac 1{2^n}.$$
Indeed to construct $l_{n+1}$, one just need to pick it  such that
$|e^{i\pi \theta l_{n+1}}+1|<\frac 1 {2^n k_n}$ as $\theta\notin \bQ$. Then
 $ |e^{i2\pi \theta k_jl_n}+1|< \frac 1{2^n}$ follows  by taking $k_j$ powers.

Similarly to get $k_{n+1}$, it suffices to choose $k$  big enough and  such that 
$|e^{i2\pi (2\theta k+\theta)}-1|<\frac 1 {2^n l_{n}}$ as $2\theta\notin \bQ$ and to put $k_{n+1}=2k+1$.

 Put $e_n=U_\theta^{k_n}V_\theta^{l_n}$. We have
$e_ne_j=e^{i2\pi \theta k_jl_n} U_\theta^{k_n+k_j}V_\theta^{l_n+l_j}$, hence
$$\| e_ne_j+ e_je_n \| \leq \frac 1{2^{n-1}},\qquad
\| e_ne_j^*+ e_j^*e_n \| \leq \frac 1{2^{n-1}}.$$
 Fix some $N$. For any $a_j\in \bC$, and $x=\sum_{j=1}^N a_j e_{n+j}$, we have the following inequalities
$$\| x \|^2\geq  \| \tau (x^*x)\|=\sum_{j=1}^N |a_j|^2$$
$$\| x \|^2\leq \|x^*x+xx^*\|\leq 2 \| \sum_{j=1}^N |a_j|^2 1\| + \frac
 {N(N-1)}{2^{n-1}} \max_j |a_j|^2.$$ We deduce that for $n$ big enough on the 
 subspace $E_N^n={\rm span\,}(e_{n+j})_{j=1}^N$ the $L_\infty$-norm 
is smaller than 2 times the $L_2$-norm.
Since the $L_\infty$-norm dominates the $L_2$-norm, $E_N^n$ is then
 2-complemented in $L_\infty(\bT^2_\theta)$ by the natural orthogonal projection; this is the Fourier multiplier
 $M_\psi$ where $\psi$ is the characteristic function of
 $\{(k_{n+j},l_{n+j})\}$.  More generally for any choice of sign
 $\varepsilon=(\varepsilon_i)$, the function
 $\psi_\varepsilon=\sum_{i=1}^N \varepsilon_i
 1_{\{(k_{n+j},l_{n+j})\}}$ defines a Fourier multiplier on $L_\infty$ of norm at most 4. Hence we can think of 
$\{(k_{n+j},l_{n+j})_{j\geq 0}\}$ as a kind of generalized hilbertian-Sidon set.

 We will use an argument from \cite{LRR} section 7 to get a
 contradiction.  We must have $\|M_{\psi_\varepsilon}\|_{cb}\leq 4c$,
 hence the sequence $(e_{n+j})_{j=1}^N$ is $4c$-completely
 unconditional in $L_\infty(\bT^2_\theta)$ and $4c$-completely
 complemented. By duality, the same holds in $L_1(\bT^2_\theta)$.

Let $f_i=e_i\tens \delta_i$ where $\delta_i$ is the basis of the operator space $R_N\cap C_N$.
Then, using that all $e_i$ are unitaries, we get that for matrices $a_i\in M_d$
$$\| \sum_i a_i\tens f_i\|_{M_d\tens L_\infty(\bT^2_\theta)\tens
  (R_N\cap C_N)}=\| \sum_i a_i\tens \delta_i \|_{M_d\tens (R_N\cap
  C_N)}=\max\{ \| \sum_i a_i^*a_i\|^{1/2},\| \sum_i
a_ia_i^*\|^{1/2}\}.$$ Moreover by complete unconditionality, $F_N={\rm
  span\,} f_i\subset L_\infty(\bT^2_\theta)\tens (R_N\cap C_N)$ is
completely complemented in $L_\infty(\bT^2_\theta)\tens (R_N\cap C_N)$
with constant $4c$.  Indeed the projection is just the average over
all possible signs $\E_{\varepsilon} (M_{\psi_\varepsilon}\tens
T_\varepsilon)$ where $T_\varepsilon(\delta_i)=\varepsilon_i
\delta_i$. By duality, we deduce that $F_N^*\approx {\rm span\,}
f_i\subset L_1(\bT^2_\theta)\hat \tens (R_N+C_n)$.  Next by the
noncommutative Khintchine inequalities (\cite{P2} Section 8.4) and
complete unconditionnality, for any matrices $a_i\in S^d_1$
$$\| \sum_i a_i\tens f_i\|_{S_1^d\hat \tens L_1(\bT^2_\theta)\hat
  \tens (R_N+ C_N)}\approx\E_{\varepsilon} \| \sum_i a_i\tens
\varepsilon_i e_i \|_{S_1^d\hat \tens L_1(\bT^2_\theta)}\approx \|
\sum_i a_i\tens e_i \|_{S_1^d\hat \tens L_1(\bT^2_\theta)}.$$ Hence
$(e_i)$ generates $R_N+C_N$ as an operator space in $L_1$. Using
duality once again, as $M_\psi$ is a projection, we get that for some
constant $C$ independent of $N$, $E_N^n$ is $C$-completely isomorphic
to $R_N\cap C_N$. This contradicts the fact that $R_N\cap C_N$ has an
injectivity constant of order
$\sqrt N$ (see \cite{Pbook}).
\end{proof}

The above argument actually shows that 
a generalized completely hilbertian-Sidon set must span $R\cap C$ as an operator space.

There is a possible variant of the proof. The set $\{e_n,n\geq 1\}$
actually generates a copy of $\ell_2$ in $L_\infty(\bT_\theta^2)$ and
is complemented. By choosing in the construction very rapidly
increasing sequences $(k_{n},l_{n})_{n\geq 1}$, one could also make
$\{(k_{n+j},l_{n+j})_{j\geq 0}\}$ a Sidon set in
$L_\infty(\bT^2)$. But infinite Sidon sets are not complemented, this also
leads to a contradiction.

Using the embedding and techniques of Corollary \ref{ex2}
with the injectivity of $L_\infty(\bT^2)$, one can also prove
that for any $\theta,\gamma\in \bR\backslash \bQ$ such that $\theta\notin \gamma \bQ$, there
is a multiplier $\phi$ such that 
$$ \|M_\phi : L_\infty(\bT_\theta^2)\to  L_\infty(\bT_\theta^2)\| <\infty \qquad \textrm{ but }
\qquad  \|M_\phi : L_\infty(\bT^2_\gamma)\to  L_\infty(\bT^2_\gamma)\| =\infty.$$

\medskip

\noindent \textbf{Acknowledgments.}  The author is supported by
ANR-2011-BS01-008-01.

\bibliographystyle{plain}

\end{document}